\documentclass[12pt]{article}
\usepackage{amsfonts}
\usepackage{amsthm}
\usepackage{url}
\usepackage{mathtools}
\usepackage[all]{xy}
\usepackage{amssymb}
\usepackage{color}
\usepackage{mathrsfs}
\usepackage{tikz}
\usepackage{hyperref}
\usetikzlibrary{calc,intersections,through,backgrounds}

\usepackage{pdfsync}
\newtheorem{thm}{Theorem}[section]

\newtheorem{lemma}[thm]{Lemma}

\newcommand{\Z}{{\mathbb Z}}

\newcommand{\Com}{{\mathbb C}}

\newcommand{\cI}{{\mathrm{c\text{--} Ind}}}
\DeclareMathOperator{\Spm}{\mathrm{m-Spec}}

\makeatletter
\let\c@equation\c@thm
\makeatother
\numberwithin{equation}{section}

\title{Generic smooth representations}

\author{Alexandre Pyvovarov}

\date{\today}

\begin{document}

\maketitle

\begin{abstract}
Let $F$ be a non-archimedean local field. In this paper we explore genericity of irreducible smooth representations of $GL_n(F)$ by restriction to a maximal compact subgroup $K$ of $GL_n(F)$. Let $(J, \lambda)$ be a Bushnell--Kutzko type for a Bernstein component $\Omega$. The work of Schneider--Zink gives an irreducible $K$-representation $\sigma_{min}(\lambda)$, which appears with multiplicity one in $\mathrm{Ind}_J^K \lambda$. Let $\pi$ be an irreducible smooth representation of $GL_n(F)$ in $\Omega$. We will prove that $\pi$ is generic if and only if $\sigma_{min}(\lambda)$ is contained in $\pi$ with multiplicity one.
\end{abstract}

\tableofcontents

\section{Introduction}

We are here concerned with the problem of understanding the genericity of irreducible smooth representations of a general linear group over a $p$-adic field. 

Let $G$ be a reductive $p$-adic group. Recall that a smooth irreducible representation $\pi$ of $G$ is called generic if $\pi$ appears in $\mathrm{Ind}_U^G \psi$ (i.e. admits a Whittaker model), where $\mathrm{Ind}$ denotes induction and $\psi$ is a nondegenerate character of a maximal unipotent subgroup $U$ of $G$.

We will start by recalling a few facts about the category of smooth representations. Let $C$ be an algebraically closed field of characteristic zero. Let  $\mathcal{R}(G)$ be the category of all smooth  $C$-representations of $G$. The Bernstein decomposition (\cite{MR771671}) expresses the category of smooth $C$-valued representations of  $G$ as the product of certain indecomposable full subcategories, called Bernstein components. Those components are parametrized by the inertial classes, whose definition we now recall. Consider the set of pairs $(M, \rho)$, with $M$ a Levi subgroup of $G$ and  $\rho$ an irreducible supercuspidal representation of $M$. We say that two pairs $(M_1, \rho_1)$ and $(M_2, \rho_2)$ are inertially equivalent if and only if there are $g \in G$ and an unramified character $\chi$ of $M_2$ such that $M_2=M_1^{g}$ and $\rho_2 \simeq \rho_1^g \otimes\chi$, where $M_1^g:=g^{-1}M_1g$ and $\rho_1^g(x) = \rho_1(gxg^{-1})$, for $x \in M_1^g$. The equivalence class of $(M, \rho)$ will be denoted by $[M,\rho]_{G}$, and is called \textit{inertial class}. The set of inertial classes will be denoted by $\mathcal{B}(G)$.

We denote by $i_{P}^{G} : \mathcal{R}(M) \longrightarrow \mathcal{R}(G)$ the normalized parabolic induction functor, where $P=MN$ is a parabolic subgroup of $G$ with Levi subgroup $M$. Let  $\Omega:=[M,\rho]_{G}$ be an inertial equivalence class, where $\rho$ is a supercuspidal representation of $M$. To $\Omega$ we may associate a full subcategory $\mathcal{R}^{\Omega}(G)$ of $\mathcal{R}(G)$, such that the representation $(\pi,V)$ is an object of $\mathcal{R}^{\Omega}(G)$ if and only if every irreducible $G$-subquotient $\pi_{0}$ of $\pi$ appears as a composition factor of $i_{P}^{G}(\rho \otimes \omega)$ for $\omega$ some unramified character of $M$ and $P$ some parabolic subgroup of $G$ with Levi factor $M$. The category $\mathcal{R}^{\Omega}(G)$ is called a Bernstein component of $\mathcal{R}(G)$. According to \cite{MR771671}, the Bernstein decomposition is written as, $\mathcal{R}(G) = \prod_{\Omega \in \mathcal{B}(G)} \mathcal{R}^{\Omega}(G)$. It follows that if we want to understand the category $\mathcal{R}(G)$, it is enough to restrict our attention to the Bernstein components. This can be done via the theory of types. This theory allows us to parametrize all the irreducible representations of $G$ up to inertial equivalence using irreducible representations of compact open subgroups of $G$. Let $J$ be a compact open subgroup of $G$ and let $\lambda$ be an irreducible representation of $J$. We say that $(J, \lambda)$ is an $\Omega$-type, if for $(\pi,V)$ a representation of $G$, the representation $(\pi,V)$ is an object of $\mathcal{R}^{\Omega}(G))$ if and only if $V$ is generated by its $\lambda$-isotypical space $V^{\lambda}$ as a $G$-representation.

Let $F$ be a local non-archimedean field. For $G=GL_n(F)$, types can be constructed (cf. \cite{MR1204652}, \cite{MR1643417} and \cite{MR1711578}) for every Bernstein component. The simplest example of a type is $(I, 1)$, where $I$ is the standard Iwahori subgroup of $G$ and $1$ is the trivial representation. In this case $\Omega = [T,1]_G$, where $T$ is the subgroup of diagonal matrices and $1$ denotes the trivial representation of $T$. We will refer to this example as the Iwahori case.

Fix $K$ a maximal compact subgroup of $G=GL_n(F)$. Given a Bushnell--Kutzko type $(J, \lambda)$ with $J$ contained in $K$, in \cite[section 6]{MR1728541} (just above Proposition~2) the authors define irreducible $K$-representations $\sigma_{\mathcal{P}}(\lambda)$, where $\mathcal{P}$ belongs to some partially ordered set (cf. \cite[section 2]{MR1728541}). One has the decomposition :
\begin{equation}\label{decomp}
\mathrm{Ind}_{J}^{K} \lambda = \bigoplus_{\mathcal{P}} \sigma_{\mathcal{P}}(\lambda)^{\oplus m_{\mathcal{P},\lambda}},
\end{equation}
\noindent where the summation runs over the same partially ordered set as above. The integers $m_{\mathcal{P},\lambda}$ are finite and we call $m_{\mathcal{P},\lambda}$ the multiplicity of $\sigma_{\mathcal{P}}(\lambda)$. Let $\mathcal{P}_{max}$ be the maximal elements and let $\mathcal{P}_{min}$ the minimal one. Define $\sigma_{max}(\lambda):=\sigma_{\mathcal{P}_{max}}(\lambda)$ and $\sigma_{min}(\lambda):=\sigma_{\mathcal{P}_{min}}(\lambda)$. Both $K$-representations $\sigma_{max}(\lambda)$ and $\sigma_{min}(\lambda)$ occur in $\mathrm{Ind}_{J}^{K} \lambda$ with multiplicity 1. In the Iwahori case those representations have a very simple description. Indeed, $\sigma_{min}(\lambda)$ is the inflation of the Steinberg representation of $GL_n(k_F)$ to $K$ and $\sigma_{max}(\lambda)$ is the trivial representation.

Having introduced the main notation of this paper we may now state our main theorem:

\begin{thm}\label{p}
	Let $\pi$ be an absolutely irreducible representation in the Bernstein component $\Omega$ and let $(J,\lambda)$ an $\Omega$-type. Then 
	$$
	\dim_C \mathrm{Hom}(\sigma_{min}(\lambda), \pi) = 
	\begin{cases}
	1\ \text{if}\ \pi \text{is a generic object of }\ \mathcal{R}^{\Omega}(G),\\
	0\ \text{otherwise.}
	\end{cases}
	$$
\end{thm}

Theorem \ref{p} shows that the representation $\sigma_{min}(\lambda)$ has a very special role. One can wonder about other $\sigma_{\mathcal{P}}(\lambda)$'s. There is a recent result by Jack Shotton in that direction. He proves \cite[Thm.3.7]{MR3769675} that by modifying the proof of  \cite[Proposition 2 Section 6]{MR1728541} and \cite[Proposition 6.5.3]{MR2656025} in the tempered case, one gets the same result in the generic case. In the author's thesis the result \cite[Thm.3.7]{MR3769675} was proven independently but with a different method. First using the theory of types of Bushnell--Kutzko, we reduce the statement to the Iwahori case. Then, in the Iwahori case, we use the results of Rogawski \cite{MR782228} on modules over Iwahori--Hecke algebra. In this case the proof relies on some easy combinatorics on partitions.

The multiplicity one statement can fail for other $\sigma_{\mathcal{P}}(\lambda)$'s. For example, consider the Iwahori case with $n=3$, i.e. $G=GL_3(F)$. Take $\pi=i_B^G(1 \otimes \chi_1 \otimes \chi_2)$, where $B$ is the subgroup of $G$ of upper triangular matrices, $1$ the trivial character and $\chi_1$, $\chi_2$ unramified characters such that $\chi_1.\chi_2^{-1} \neq |.|^{\pm 1}$. Then, writing $\sigma_{2,1}$ for the summand of $\mathrm{Ind}_I^K 1$ corresponding to the partition $(2,1)$ (see section 2), one can easily verify that $\dim \mathrm{Hom}_K(\sigma_{2,1}, \pi)=2$.

Let us say a few words about the proof of Theorem \ref{p}. First we use one of the main results of  \cite{MR1711578}, which asserts that the Hecke algebra $\mathcal{H}(G,\lambda)$ is naturally isomorphic to a tensor product of affine Hecke algebras of type A. Moreover it is shown in  \cite{MR1204652} that any Hecke algebra of a simple type is isomorphic to an affine Hecke algebra of type A. In this manner we can reduce the statement about irreducible representations of general type to the Iwahori case. It was pointed out to us, recently, by Peter Schneider that the Iwahori case was already treated by \cite{MR1915088}. This allowed us to simplify a little the original proof in the author's thesis.

Finally let us observe that to the best of our knowledge Theorem \ref{p} and \cite[Thm.3.7]{MR3769675} do not have an analogue for all reductive groups, because the crucial ingredient in the proofs is the tensor product decomposition of the Hecke algebra $\mathcal{H}(G,\lambda)$ and the existence of types, proven by Bushnell--Kutzko in \cite{MR1711578}. Indeed results of \cite{MR1204652}, \cite{MR1643417} and \cite{MR1711578} allow us to transfer the general situation to the Iwahori case, where the proofs are simpler. However we believe that those results should generalize easily to reductive groups with $A_n$ root system. It would be interesting to investigate the case of other reductive groups.

\subsection*{Notation}

For an arbitrary local non-archimedean field $L$, let $\mathcal{O}_L$ be its ring of integers and $k_{L}$ the residue field. We also choose a uniformizer $\varpi_L \in \mathcal{O}_L$. From now on fix $F$ a local non-archimedean field and $G=GL_n(F)$.

Recall that all the representations have their coefficients in an algebraically closed field $C$ of characteristic zero. Assume that $C$ has the same cardinality as the complex numbers $\Com$. Fix an isomorphism $\iota : C \rightarrow \Com$. Let $\tilde{G}$ be some $p$-adic group. A character $\chi: \tilde{G} \rightarrow C$ is defined by $\chi = \iota^{-1}(\iota \circ \chi)$, where $\iota \circ \chi$ is a character in a usual sense.

We are given an inertial class $\Omega=[M,\rho]_{G}$, where $\rho$ is a supercuspidal representation of $M$ and an $\Omega$-type $(J, \lambda)$ with $J \subset K$ a compact open subgroup of $G$. Write $\mathfrak{Z}_{\Omega}$ for the centre of the category $\mathcal{R}^{\Omega}(G)$. Recall that the centre of a category is the ring of endomorphisms of the identity functor. For example the centre of the category $\mathcal{H}(G,\lambda)\text{-Mod}$ is $Z(\mathcal{H}(G,\lambda))$, where $Z(\mathcal{H}(G,\lambda))$ is the centre of the ring $\mathcal{H}(G,\lambda)$. 

The representations of a Bernstein component can be seen as modules over Hecke algebra. Let $\mathcal{R}_{\lambda}(G)$ be a full subcategory of $\mathcal{R}(G)$ such that $(\pi,V)$ is an object of $\mathcal{R}_{\lambda}(G)$ if and only if $V$ is generated by $V^{\lambda}$ (the $\lambda$-isotypical component of $V$) as $G$-representation. Define $\mathcal{H}(G,\lambda):= \mathcal{H}(G,J,\lambda):=\mathrm{End}_{G}(\cI_J^G \lambda)$, the Hecke algebra of the type $(J,\lambda)$. Then for any $\Omega$-type $(J,\lambda)$, by \cite[Theorem 4.2 (ii)]{MR1643417},  the functor:
$$\begin{array}{ccccc}
\mathfrak{M}_{\lambda} & : & \mathcal{R}_{\lambda}(G) & \to & \mathcal{H}(G,\lambda)\text{-Mod} \\
& & \pi &  \mapsto & \mathrm{Hom}_{J}(\lambda, \pi) = \mathrm{Hom}_{G}(\cI_J^G \lambda, \pi)\\
\end{array}$$
\noindent is an equivalence of categories. Since  $(J,\lambda)$ is an $\Omega$-type, we have $\mathcal{R}^{\Omega}(G)=\mathcal{R}_{\lambda}(G)$.

As in \cite{MR1728541} a \textit{partition} is a function $P:\Z_{\geq 1} \rightarrow \Z_{\geq 0}$ with finite support; we say that $P$ is a partition of an integer $k:=\sum_{n=1}^{+\infty} P(n).n$. Usually a partition $P$ of $k$ is represented by a sequence $(m_1,\ldots,m_k)$, with $m_1 \geq \ldots \geq m_k \geq 0$ and $m_1+\ldots+m_k=k$, where one omits the zeroes from that list. The integers $m_i$ are related to $P$ as follows, $m_k=P(k)$, $m_{k-1}=P(k)+P(k-1)$,..., $m_1=P(k)+\ldots+P(1)$. We define a partial ordering on the set $\mathbb{P}$ of partitions as follows. We write $\lambda =(\lambda_{1},\ldots, \lambda_{k}) \geqq \mu =(\mu_{1},\ldots, \mu_{k})$ if and only if $\sum_{i=1}^{j} \lambda_{i} \leq \sum_{i=1}^{j} \mu_{i}$ for all integers $j$. The smallest partition for this partial order is $(k)$ and the biggest is $(1,\ldots,1)$ ($k$ times $1$). This is the opposite of the usual order on partitions (\cite[Chapter 5, Section 5.1.4]{MR3077154}).

As in \cite[section 2]{MR1728541}, let $\mathcal{C}$ be a system of representatives for the irreducible supercuspidal representations of any $GL_k(F)$ ($ k \in \Z_{\geq 1}$) up to unramified twist. A partition-valued function is a function $\mathcal{P}: \mathcal{C} \rightarrow \mathbb{P}$ with finite support. The set of partition-valued functions is partially ordered with respect to the partial ordering on partitions defined in the paragraph above by setting $\mathcal{P}\leq \mathcal{P}'$ if and only if $\mathcal{P}(\tau) \leq \mathcal{P}'(\tau)$, $\forall \tau \in \mathcal{C}$. Choose a partition-valued function $\mathcal{P}^{min}$ which is minimal for this partial ordering as in \cite{MR1728541}. Recall the decomposition (\ref{decomp}) from the introduction:
\[\mathrm{Ind}_{J}^{K} \lambda = \bigoplus_{\mathcal{P}} \sigma_{\mathcal{P}}(\lambda)^{\oplus m_{\mathcal{P},\lambda}},\]
\noindent where the summation runs over partition-valued functions.
From now on let $\sigma_{min}(\lambda) := \sigma_{\mathcal{P}^{min}}(\lambda)$ with the notations of section 6 in \cite{MR1728541}.

Let me introduce some further notation. Denote by $W$ the vector space on which the representation $\lambda$ is realized. Next, let $(\check{\lambda},W^{\vee})$ denote the contragradient of $(\lambda,W)$. Then by \cite[(2.6)]{MR1711578}, the Hecke algebra $\mathcal{H}(G, \lambda)$ can be identified with the space of compactly supported functions $f : G \longrightarrow \mathrm{End}_{C}(W^{\vee})$ such that $f(j_1.g.j_2) = \check{\lambda}(j_1)\circ f(g) \circ \check{\lambda}(j_2)$, with $j_1$, $j_2 \in J$ and $g \in G$ and the multiplication of two elements $f_1$ and $f_2$ is given by convolution:
$$f_1*f_2(g) = \int_G f_1(x)\circ f_2(x^{-1}g) dx.$$
For $u \in \mathrm{End}_C(W^{\vee})$, we write $u^{\vee} \in \mathrm{End}_{C}(W)$ for the transpose of $u$ with respect of the canonical pairing between $W$ and $W^{\vee}$. This gives $(\check{\lambda}(j))^{\vee} = \lambda(j)$, for $j \in J$. For $f \in \mathcal{H}(G,\lambda)$, define $\check{f} \in \mathcal{H}(G, \check{\lambda})$, by $\check{f}(g) = f(g^{-1})^{\vee}$, for all $g \in G$.

\section{Simple types}\label{M.8}

Let $E=F[\beta]$ be a finite field extension of $F$. Define $R=n/[E:F]$. Let $(J, \lambda)$ a simple type in $G$, where $J$ is a compact open subgroup in $G$ and $\lambda= \kappa \otimes \sigma$ with $\kappa$ a $\beta$-extension and $\sigma$ the inflation of $\tau \otimes\ldots\otimes \tau$  ($e$-times), where $\tau$ a cuspidal representation of $GL_{f}(k_{E})$,  and we have $R=ef$.

Let $W=S_{e}$ a symmetric group in $e$ elements, and let $S$ be the subset of $W$ of all the transpositions $(i,i+1)$. Then $(W,S)$ be a Coxeter group. The Hecke algebra of $(W,S)$, denoted $\mathcal{H}_{W}$, is spanned by elements $T_{w}$, $w\in W$, subject to relations:
$$T_{x}T_{y}=T_{xy} \mbox{  if  } l(xy)=l(x)+l(y)$$
$$T_{s}^{2} = (q-1)T_{s}+q \mbox{ for all } s\in S,$$ 
\noindent where $l$ denotes the length of reduced decomposition of an elements in $W$.

In this section $\overline{B}:=B(k_{E})$ is the Borel subgroup of $\overline{G}_{e}=GL_{e}(k_{E})$ and let $\overline{G}=GL_{R}(k_{E})$.  We will always identify $w\in W$ with a matrix in $\overline{G}_{e}$ or with a matrix in $\overline{G}$, depending on the context.

Let $\overline{P}$ be a subgroup of $\overline{G}$ consisting of is upper triangular matrices by blocs with bloc sizes $f\times f$. Let $\phi_{w} \in \mathcal{H}(\overline{G},\sigma)$ is null outside $\overline{P}w\overline{P}$ such that $\phi_{w}(p_{1}wp_{2})=\sigma(p_{1})\circ \phi_{w}(w) \circ \sigma(p_{2})$ and $\phi_{w}(w)(y_{1}\otimes\ldots\otimes y_{e})= y_{w(1)}\otimes\ldots\otimes y_{w(e)}$. The homomorphism of Hecke algebras, as in (5.6.1) \cite{MR1204652}:
$$\begin{array}{ccccc}
\Psi & : & \mathcal{H}_{W} & \to & \mathcal{H}(\overline{G},\sigma) \\
& & T_{w} &  \mapsto & \phi_{w} \\
\end{array}$$ 

\noindent is actually an isomorphism according to Theorem 5.1 in Chapter 1 \cite{MR821216}. In fact one can carry out a calculation to prove that  $\phi_{w}$ are generators of $\mathcal{H}(\overline{G},\sigma)$ and they satisfy the same relations as $T_{w}$ in $\mathcal{H}_{W}$.

We have the following isomorphisms of Hecke algebras:
$$\mathcal{H}_{W} \simeq \mathrm{End}_{\overline{G}_{e}}(\mathrm{Ind}_{B}^{\overline{G}_{e}}1)$$
\noindent and 
$$\mathcal{H}(\overline{G},\sigma) \simeq \mathrm{End}_{\overline{G}}(\mathrm{Ind}_{\overline{P}}^{\overline{G}}\sigma)$$

Let $\mathscr{M}(\overline{G}_{e})$ be the category of $\overline{G}_{e}$-representations and $\mathscr{M}_{\nu}(\overline{G}_{e})$ the full subcategory of $\mathscr{M}(\overline{G}_{e})$ of all $\overline{G}_{e}$-representations whose irreducible constituents all have cuspidal support $\nu=\tau \otimes\ldots\otimes\tau$. Define:
$$\begin{array}{ccccc}
&  & \mathscr{M}_{1}(\overline{G}_{e}) & \to & \mathcal{H}_{W} \\
& & \pi &  \mapsto & \mathrm{Hom}_{\overline{G}_{e}}(\mathrm{Ind}_{B}^{\overline{G}_{e}}1, \pi) \\
\end{array}$$
$$\begin{array}{ccccc}
\Psi' & : & \mathcal{H}_{W}-Mod & \to & \mathcal{H}(\overline{G},\sigma)-Mod \\
& & M &  \mapsto & M \otimes_{\mathcal{H}_{W}} \mathcal{H}(\overline{G},\sigma) \\
\end{array}$$
$$\begin{array}{ccccc}
&  & \mathcal{H}(\overline{G},\sigma)-Mod & \to & \mathscr{M}_{\nu}(\overline{G}) \\
& & M &  \mapsto & M \otimes_{\mathcal{H}(\overline{G},\sigma)} \mathrm{Ind}_{\overline{P}}^{\overline{G}}\sigma  \\
\end{array}$$
Let  $\overline{H}_e : \mathscr{M}_{1}(\overline{G}_{e}) \to \mathscr{M}_{\nu}(\overline{G})$ the composition of these 3 functors.

First notice that $\overline{H}_e(\mathrm{Ind}_{B}^{\overline{G}_{e}}1) = \mathrm{Ind}_{\overline{P}}^{\overline{G}}\sigma$. Let $\overline{Q}$ be any standard parabolic of $\overline{G}_{e}$. We obtain a standard parabolic $\widehat{Q}$ of $\overline{G}$ from $\overline{Q}$ by enlarging each entry of $\overline{Q}$ to a bloc of size $f\times f$. We use the same convention for Levi subroups.

\begin{lemma}\label{1.13}
	Let $Q$ be a standard parabolic of $\overline{G}_{e}$ and $\tilde{Q}$ as above, a standard parabolic of $\overline{G}$. Then:
	$$\overline{H}_e(\mathrm{Ind}_{\overline{Q}}^{\overline{G}_{e}}1)= \mathrm{Ind}_{\widehat{Q}}^{\overline{G}}\sigma$$
\end{lemma}

\begin{proof} Let $W_{\overline{Q}}$ be the parabolic subgroup of $W$ associated to $\overline{Q}$. We have
	$$\overline{H}_e(\mathrm{Ind}_{\overline{Q}}^{\overline{G}_{e}}1) = \mathrm{Hom}_{\overline{G}_{e}}(\mathrm{Ind}_{\overline{B}}^{}1, \mathrm{Ind}_{\overline{Q}}^{\overline{G}_{e}}1) \otimes_{\mathcal{H}(\overline{G},\sigma)} \mathrm{Ind}_{\overline{P}}^{\overline{G}}\sigma$$
	$$=\bigoplus_{w \in W/W_{\overline{Q}}} \mathrm{Hom}_{\overline{B}^{w}\cap \overline{Q}}(1, 1^{w})\otimes_{\mathcal{H}(\overline{G},\sigma)} \mathrm{Ind}_{\overline{P}}^{\overline{G}}\sigma$$
	
	We identify, as usual, $\mathrm{Hom}_{B^{w}\cap Q}(1, 1^{w})$ with the set of functions in $\mathcal{H}_{W}$ supported on $\overline{B}w\overline{Q}$. Via the isomorphism $\Psi$ of Hecke algebras, the set functions in $\mathcal{H}_{W}$ supported on $\overline{B}w\overline{Q}$ is in bijection with the set of functions in $\mathcal{H}(\overline{G},\sigma)$ supported on $\overline{P}w\widehat{Q}$ and this set is indetified with the intertwining set  $\mathrm{Hom}_{\overline{P}^{w}\cap \widehat{Q}}(\sigma, \sigma^{w})$. It follows, that:
	$$\overline{H}_e(\mathrm{Ind}_{Q}^{\overline{G}_{e}}1)\simeq \bigoplus_{w \in W/W_{\overline{Q}}} \mathrm{Hom}_{\overline{P}^{w}\cap \tilde{Q}}(\sigma, \sigma^{w}) \otimes_{\mathcal{H}(\overline{G},\sigma)} \mathrm{Ind}_{\overline{P}}^{\overline{G}}\sigma$$ 
	$$= \mathrm{Hom}_{\overline{G}}(\mathrm{Ind}_{\overline{P}}^{\overline{G}}\sigma, \mathrm{Ind}_{\tilde{Q}}^{\overline{G}}\sigma)\otimes_{\mathcal{H}(\overline{G},\sigma)} \mathrm{Ind}_{\overline{P}}^{\overline{G}}\sigma$$
	\noindent The result follows.
\end{proof}

Let $st(\tau,e)$ be a representation of $\overline{G}_{fe}$, defined as the unique nondegenerate irreducible representation with cuspidal support $\tau \otimes\ldots\otimes \tau$  ($e$-times). Since $\overline{H}_e$ is exact, $\overline{H}_e(st(1,e))=\frac{\overline{H}_e(\mathrm{Ind}_{B}^{\overline{G}_{e}}1)}{\sum_{Q \varsupsetneq B}\overline{H}_e(\mathrm{Ind}_{Q}^{\overline{G}_{e}}1)}$ and by previous lemma, $\overline{H}_e(st(1,e))=st(\tau,e)$.

\begin{lemma}\label{1.14}
	Let $\overline{M'}$ be the Levi subgroup of $\overline{Q}$. The following diagram commutes:
	$$\xymatrixcolsep{5pc}\xymatrix{
		\mathscr{M}_{1}(\overline{G}_{e})  \ar[r]^{\overline{H}_e} &\mathscr{M}_{\nu}(\overline{G}) \\
		\mathscr{M}_{1}(\overline{M'}) \ar[u]^{\mathrm{Ind}_{\overline{Q}}^{\overline{G}_{e}}} \ar[r]^{\overline{H}_{\overline{M'}}} &\mathscr{M}_{\nu_{\widehat{M'}}}(\widehat{M'}) \ar[u]^{\mathrm{Ind}_{\widehat{Q}}^{\overline{G}}}
	}$$
	\noindent where the horizontal arrows are an equivalence of categories, and $\nu_{\widehat{M'}}$ denotes the restriction of the cuspidal support $\nu$ to $\widehat{M'}$.
\end{lemma}

\begin{proof}  It is enough to check that this diagram commutes for every irreducible representation of $\mathscr{M}_{1}(\overline{M'})$. By definition, the irreducible representations of $\mathscr{M}_{1}(\overline{M'})$ are just unramified characters of $L$. Moreover, by Lemma \ref{1.13} we have that $\overline{H}_e(\mathrm{Ind}_{\overline{Q}}^{\overline{G}_{e}}1)= \mathrm{Ind}_{\widehat{Q}}^{\overline{G}}\overline{H}_{\overline{M'}}(1)$. The same identity holds if $1$ is replaced any unramified character of $L$, by the same argument as in Lemma \ref{1.13}.
\end{proof}

We identify the partitions and partition valued functions. Let $\tilde{\mathcal{P}}$ be the partition $(e_{1}f,\ldots,e_{k}f)$ of $R$ associated to parabolic subgroup $\tilde{Q}$ and $\mathcal{P}$ the partition $(e_{1},\ldots,e_{k})$ of $e$ associated to parabolic subgroup $Q$. Define $\pi(\tau,\tilde{\mathcal{P}})=\mathrm{Ind}_{\tilde{Q}}^{\overline{G}} st(\tau,e_{1}) \otimes \ldots\otimes st(\tau,e_{k})$ and $\sigma(\tau,\tilde{\mathcal{P}})$ the representation of $\overline{G}$ that occurs in $\pi(\tau,\tilde{\mathcal{P}})$ with multiplicity 1 and not in $\pi(\tau,\mathcal{Q})$ if $\mathcal{Q}> \tilde{\mathcal{P}}$.

\begin{lemma}\label{1.15}
	Let $Q$ be a standard parabolic of $\overline{G}_{e}$ and $\tilde{Q}$ as above, a standard parabolic of $\overline{G}$. Then:
	$$\overline{H}_e(\sigma(1,\mathcal{P}))= \sigma(\tau,\tilde{\mathcal{P}})$$
\end{lemma}
\begin{proof} By previous lemma we have that:
	$$\overline{H}_e(\pi(1,\mathcal{P})) = \mathrm{Ind}_{\tilde{Q}}^{\overline{G}} F_{L}( st(1,e_{1}) \otimes \ldots\otimes st(1,e_{k}) )$$
	$$ = \mathrm{Ind}_{\tilde{Q}}^{\overline{G}} F_{e_{1}}( st(1,e_{1})) \otimes \ldots\otimes F_{e_{k}}(st(1,e_{k})) = \pi(\tau,\tilde{\mathcal{P}})$$
	
	\noindent Since $\overline{H}_e$ is exact:
	$$\overline{H}_e(\sigma(1,\mathcal{P})) = \frac{\overline{H}_e(\pi(1,\mathcal{P})}{\sum_{\mathcal{Q} < \mathcal{P}}\overline{H}_e(\pi(1,\mathcal{Q})} =  \frac{\pi(\tau,\tilde{\mathcal{P}})}{\sum_{\mathcal{Q} < \mathcal{P}}\pi(\tau,\tilde{\mathcal{Q}})}=\sigma(\tau,\tilde{\mathcal{P}})$$
\end{proof}

Let $\pi$ be an irreducible representation containing a simple type $(J,\lambda)$. In this case $\Omega = [GL_{r}(F)^{e}, \omega \otimes \ldots \otimes \omega]_{G}$ where the tensor product $\rho : = \omega \otimes \ldots \otimes \omega$  is taken $e$ times and $\omega$ is a supercuspidal representation of $GL_{r}(F)$. According to the description of Hecke algebras in section (5.6) of \cite{MR1204652} there is a support preserving isomorphism of Hecke algebras $\mathcal{H}(G_{L}, I_{L}, 1)  \simeq \mathcal{H}(G, J, \lambda)$, where $L$ is an extension of $F$ (denoted by $K$ in \cite{MR1204652}), $G_{L}=GL_{e}(L)$ with $I_{L}$ the standard Iwahori subgroup of $G_{L}$. We denote by $K_{L}$ a maximal compact subgroup of $G_{L}$ containing $I_L$. 

We will recall now the results on supercuspidal representations from chapter~6 of \cite{MR1204652} and describe the general form of the supercuspidal representation $\omega$ of $G_0=GL_{r}(F)$. The representation $\omega$ contains a maximal simple type $(J_0, \lambda_0)$. This means that there are a finite extension $E$ of $F$ and a uniquely determined representation $\Lambda_0$ of $E^{\times}J_0$ such that $\omega = \cI_{E^{\times}J_0}^{G_0} \Lambda_0$ and $\Lambda_0|J_0 = \lambda_0$.  Let $f=\frac{n}{e[E:F]}$. According to  \cite[Proposition 5.5.14]{MR1204652}, the extension $L$ considered in the previous is unramified extension of degree $f$ of $E$. A special case of the support preserving isomorphism in the previous paragraph is the support preserving isomorphism $\Phi_{1}: \mathcal{H}(G_0, J_0, \lambda_0) \simeq \mathcal{H}(L^{\times}, \mathcal{O}_{L}^{\times}, 1)$ sending a function supported on $J_0\varpi_{E}J_0=\varpi_{E}J_0$ to a function supported on $\varpi_{E} \mathcal{O}_{L}^{\times}$, where $\varpi_{E}$ a uniformizer of both $E$ and $L$. Further we observe that the unramified characters of $G_0$ are determined by the image of $\varpi_{E}$, as are unramified characters of $L^{\times}$. Therefore we may and we will identify the unramified characters of $G_0$ with the unramified characters of $L^{\times}$.

The representation $\pi$ is a Langlands quotient of the form $Q(\Delta_{1},\ldots,\Delta_{s})$ (cf. \cite[Section 1.2 Theorem 1.2.5]{MR1265559}) such that for $i < j$ the segment $\Delta_{i}$ does not precede $\Delta_{j}$ (cf. \cite[Section 1.2 Definition 1.2.4]{MR1265559}). After twisting $\pi$ by some unramified character we may assume that all the segments are of the form $\Delta_{i}=[\omega(\alpha_{i}),\omega(\alpha_{i}+e_{i}-1)]$, where $\alpha_{i}\in C$ and $e_{i}$ an integer such that $\sum_{i=1}^{s} e_{i} = e$. Here the notation $\omega(\alpha_{i})$ means that $\omega(\alpha_{i}):=\omega \otimes \iota^{-1}(|\det|^{\iota(\alpha_{i})})$, the norm $|.|$ is viewed as taking values in $q^{\Z}\subset\Com$. 

Let $P$ be a standard parabolic of $G$ containing $M=GL_r(F)^e$ and let $B_L$ be a Borel subgroup of $G_L$ containing $T_L = (L^{\times})^{e}$.

According to \cite[Theorem 7.6.20]{MR1204652}, the diagram
\begin{equation}\tag{D1}\label{D1}
\xymatrix{
	\mathcal{H}(G, J, \lambda)  \ar[r]^{\Phi} &\mathcal{H}(G_{L}, I_{L}, 1)\\
	\mathcal{H}(M, J_{M}, \lambda_{M}) \ar[u]^{t_P} \ar[r]^{\Phi_{1}^{\otimes e}} &\mathcal{H}(T_{L}, T_{L}^{\circ}, 1) \ar[u]^{t_{B_L}}}
\end{equation}
\noindent is commutative, where the horizontal arrows are support preserving isomorphisms and $\lambda_{M} = \lambda_{0}\otimes \ldots \otimes \lambda_{0}$ ($e$ times), $J_{M} = (J_{0})^{e}$, $T_{L}=(L^{\times})^{e}$ and $T_{L}^{\circ} = (\mathcal{O}_L^{\times})^{e}$.  In \cite{MR1204652}, the horizontal isomorphisms in the commutative diagram above are given in the other direction. For $t:A\rightarrow A'$ a morphism $C$-algebras we write $t_{\ast} : A'-\mathrm{Mod}\longrightarrow A-\mathrm{Mod}$ for the induced functor given $\mathrm{Hom}_{A'}(A,\bullet)$.  The diagram above produces the following commutative diagram:
\begin{equation}\tag{D2}\label{D2}\xymatrix{
	\mathcal{R}_{\lambda}(G)  \ar[r]^-{M_{\lambda}} &\mathcal{H}(G, J, \lambda)\text{-Mod}  \ar[r]^{\Phi_{\ast}} &\mathcal{H}(G_{L}, I_{L}, 1)\text{-Mod}\ar[r]^-{T_{\lambda}} &\mathcal{R}_{1}(G_{L})\\
	\mathcal{R}_{\lambda_{M}}(M) \ar[u]^-{i_{P}^{G}} \ar[r]_-{M_{\lambda_M}} &\mathcal{H}(M, J_{M}, \lambda_{M})\text{-Mod} \ar[u]^{(t_{P})_{\ast}}  \ar[r]^{(\Phi_{1}^{\otimes e})_{\ast}} &\mathcal{H}(T_{L}, T_{L}^{\circ}, 1)\text{-Mod} \ar[u]^{(t_{B_L})_{\ast}} \ar[r]^-{T_1} &\mathcal{R}_{1}(T_{L})\ar[u]_-{i_{B_{L}}^{G_{L}}}}
\end{equation}
\noindent where the horizontal arrows are equivalences of categories, $T_{\lambda} = \bullet \otimes_{\mathcal{H}(G_{L}, I_{L}, 1)}\mathrm{c\text{--} Ind}_{I_{L}}^{G_{L}} 1$, $T_1=\bullet \otimes_{\mathcal{H}(T_{L}, T_{L}^{\circ}, 1)}\mathrm{c\text{--} Ind}_{T_{L}^{\circ}}^{T_{L}} 1$, $M_{\lambda}=\mathrm{Hom}_{J}(\lambda,\bullet)$ and $M_{\lambda_M}=\mathrm{Hom}_{J_{M}}(\lambda_{M},\bullet)$. The first and second squares are commutative as a consequence of \cite[Corollary 8.4]{MR1643417}, the commutativity of the middle square follows from diagram (\ref{D1}). Let $H$ be the composition of all the top horizontal arrows. Hence the functor $H:\mathcal{R}_{\lambda}(G) \longrightarrow \mathcal{R}_{1}(G_{L})$ from above is an equivalence of categories.

\begin{lemma}\label{1.20}
	We have $H(i_{P}^{G} \rho)=i_{B_{L}}^{G_{L}} 1$.
\end{lemma}

\begin{proof}It follow from the commutative diagram above that: 
	$$\Phi_{\ast} (\mathrm{Hom}_{J}(\lambda,i_{P}^{G}(\rho)))\otimes_{\mathcal{H}(G_{L}, I_{L}, 1)}\mathrm{c\text{--} Ind}_{I_{L}}^{G_{L}} 1$$
	$$= i_{B_{L}}^{G_{L}}((\Phi_{1}^{\otimes e})_{\ast}(\mathrm{Hom}_{J_{M}}(\lambda_{M},\rho))\otimes_{\mathcal{H}(T_{L}, T_{L}^{\circ}, 1)}\mathrm{c\text{--} Ind}_{T_{L}^{\circ}}^{T_{L}} 1)$$
	Observe that the representation $\mathrm{c\text{--} Ind}_{T_{L}^{\circ}}^{T_{L}} 1$ is canonically a rank 1 free $\mathcal{H}(T_{L}, T_{L}^{\circ}, 1)$-module. This observation allows us to simplify the right hand side. Recall that $\rho : = \omega \otimes \ldots \otimes \omega$.
	
	Since $(J, \lambda)$ is a simple type, $\lambda_{M} = \lambda_{0}\otimes \ldots \otimes \lambda_{0}$ ($e$ times), $J_{M} = (J_{0})^{e}$ and  $(J_{0}, \lambda_{0})$ is a maximal simple type for the supercuspidal representation $\omega$, we have:
	$$\mathrm{Hom}_{J_{M}}(\lambda_{M},\rho) \simeq \mathrm{Hom}_{J_{M}}(\lambda_{0}\otimes \ldots \otimes \lambda_{0},\omega | J_{0} \otimes \ldots \otimes \omega| J_{0})$$ 
	$$\simeq \mathrm{Hom}_{J_{0}^{e}}(\lambda_{0}\otimes \ldots \otimes \lambda_{0},\lambda_{0}\otimes \ldots \otimes \lambda_{0})\simeq \mathrm{Hom}_{J_M}(\lambda_M, \lambda_M),$$
	\noindent where have used that $\lambda_{0}$ occurs in $\omega$ with multiplicity one. Now notice that $\mathrm{Hom}_{J_M}(\lambda_M, \lambda_M)$ is the subspace of functions in $\mathcal{H}(M, J_{M}, \lambda_{M})$ supported on $J_M$. By our choice of $\lambda_0$, $\omega$ and $\Phi_{1}$, the support preserving isomorphism $\Phi_{1}^{\otimes e}$ maps this space isomorphically onto the space of functions in $\mathcal{H}(T_{L}, T_{L}^{\circ}, 1)$ supported on $T_{L}^{\circ}$. It follows that $\Phi_{1}^{\otimes e}(\mathrm{Hom}_{J_{M}}(\lambda_{M},\rho)) = \mathrm{Hom}_{T_{L}^{\circ}}(1, 1)$. Thus, the representation $\Phi_{1}^{\otimes e}(\mathrm{Hom}_{J_{M}}(\lambda_{M},\rho))\otimes_{\mathcal{H}(T_{L}, T_{L}^{\circ}, 1)}\mathrm{c\text{--} Ind}_{T_{L}^{\circ}}^{T_{L}} 1$ is a trivial character of $T_{L}$. Then an object $i_{P}^{G} \rho$ in $\mathcal{R}_{\lambda}(G)$ corresponds to an object $i_{B_{L}}^{G_{L}} 1$ in $\mathcal{R}_{1}(G_{L})$.
	
\end{proof}

\begin{lemma}\label{1.21}
	$H:\mathcal{R}_{\lambda}(G) \longrightarrow \mathcal{R}_{1}(G_{L})$ is compatible with twisting by characters
\end{lemma}

\begin{proof}
	Consider the representation $i_{P}^{G}((\omega \otimes \chi_1)\otimes \ldots (\omega \otimes \chi_e))$, where $\chi_1,\ldots,\chi_e$ are some unramified characters of $G_0$. Let $\rho' = (\omega \otimes \chi_1)\otimes \ldots (\omega \otimes \chi_e) = \rho \otimes \chi$, where $\chi$ is an unramified character of $M$. According to \cite[page 591]{MR1643417}  the action of $\mathcal{H}(M, J_{M}, \lambda_{M})$ on $\mathrm{Hom}_{J_M}(\lambda_M, \rho')$ is given by
	$$f\cdot\phi(w) = \int\limits_{M} \rho'(g)\phi(\check{f}(g^{-1})w)dg$$
	\noindent where $f \in \mathcal{H}(M, J_{M}, \lambda_{M})$,  $\phi \in \mathrm{Hom}_{J_M}(\lambda_M, \rho')$ and $w$ is a vector in the underlying vector space of $\lambda_M$. We want to understand the compatibility of this action with twisting and support preserving isomorphisms of Hecke algebras. Since $f$ has compact support, without loss of generality we may assume that $f$ is supported on $J_M m J_M$, for some $m \in M$. The element $m$ is a block-diagonal matrix with $e$ blocks. Without loss of generality we may assume that each block is some power of the uniformizer $\varpi_{E}$. We normalize measures on $G_0$ and on $L^{\times}$ such that $J_0$ and $\mathcal{O}^{\times}$ have both volume 1. Then taking the induced product measures on $M$ and $T_L$, we see that $\int_{J_M}dj=1$ and $\int_{T_{L}^{\circ}}dt=1$. Each block in the matrix $m$ normalizes $J_0$, hence $m$ normalizes $J_M$ and $J_M mJ_M =mJ_M$.  It follows that
	$$f \cdot\phi(w)  = \int\limits_{j \in J_M}  \rho'(mj)\phi(f(mj)^{\vee}w) dj.$$
	By definition we have: 
	$$f(mj)^{\vee} = (f(m)\check{\lambda}_M(j))^{\vee} = \lambda_M(j^{-1}).f(m)^{\vee}.$$
	Moreover $\phi$ is $J_M$-equivariant, thus  
	$$\phi(\lambda_M(j^{-1})f(m)^{\vee}) = \rho'(j^{-1}).\phi(f(m)^{\vee}).$$
	This simplifies the integral above:
	$$\int\limits_{j \in J_M} \rho'(m).\phi(f(m)^{\vee} w) dj=\rho'(m).\phi(f(m)^{\vee}w)=\chi(m).\rho(m).\phi(f(m)^{\vee}w).$$
	The expression above is compatible with the support preserving isomorphism $\Phi_{1}^{\otimes e}$, in a sense that
	$$\Phi_{1}^{\otimes e} (\chi(m).\rho(m).\phi(f(m)^{\vee}\bullet) )= \chi(m). \Phi_{1}^{\otimes e}(\phi)(\Phi_{1}^{\otimes e}(f)(m)^{\vee}\bullet),$$
	\noindent where $m$ is naturally seen as an element of $T_L$ because its diagonal blocks are some powers of the uniformizer $\varpi_{E}$ and $\chi$ is seen as unramified character of $T_L$. This is, of course, compatible with the same computation of the integral replacing $ \mathcal{H}(M, J_{M}, \lambda_{M})$ by $\mathcal{H}(T_{L}, T_{L}^{\circ}, 1)$ and   $\mathrm{Hom}_{J_M}(\lambda_M, \rho')$ by $\mathrm{Hom}_{T_{L}^{\circ}}(1, \chi)$.
\end{proof}

\begin{lemma}\label{1.22}
	The $H:\mathcal{R}_{\lambda}(G) \longrightarrow \mathcal{R}_{1}(G_{L})$	preserves the segments.
\end{lemma}

\begin{proof}
	
	We know that $\pi= Q(\Delta_{1},\ldots,\Delta_{s})$ is an irreducible subquotient of $i_{P}^{G}((\omega \otimes \chi_1)\otimes \ldots (\omega \otimes \chi_e))$, where $\chi_1,\ldots,\chi_e$ are some unramified characters of $G_0$.
	Then by the equivalence of categories described in the diagram (\ref{D2}), $H(\pi)$ is an irreducible subquotient of $H(i_{P}^{G}((\omega \otimes \chi_1)\otimes \ldots \otimes(\omega \otimes \chi_e)))$. However by Lemma \ref{1.21}, we have $H(\pi)$ is an irreducible subquotient of $H(i_{P}^{G}((\omega \otimes \chi_1)\otimes \ldots \otimes(\omega \otimes \chi_e)))=i_{B_{L}}^{G_{L}} (\chi_1 \otimes \ldots \otimes \chi_e)$. Let now $\Delta = [\omega(\alpha),\omega(\alpha+e-1)]$, a segment in $G$, where $\alpha$ is some scalar. Then the commutative diagram above shows that the $G$-representation $i_{P}^{G}(\Delta)$ corresponds to the $G_{L}$-representation $H(i_{P}^{G}(\Delta))=i_{B_{L}}^{G_{L}}(\Delta_{L})$, where $\Delta_{L} = [1(\alpha),1(\alpha+e-1)]$ is a segment in $G_{L}$ and $1$ is the trivial character of $L^{\times}$. We know that  $i_{P}^{G}(\Delta)$ and $i_{B_L}^{G_L}(\Delta_L)$, admit unique irreducible quotients $Q(\Delta)$ and $Q(\Delta_{L})$ respectively, so $H(Q(\Delta))= Q(\Delta_{L})$. Let $\tilde{P}=\tilde{M}\tilde{N}$, with Levi subgroup $\tilde{M}=GL_{re_1}(F) \times \ldots \times GL_{re_s}(F)$ and unipotent radical $\tilde{N}$, be a standard parabolic containing $P$ which is adapted to the segment decomposition of $\pi=Q(\Delta_{1},\ldots,\Delta_{s})$, so that $\pi$ is a quotient of $i_{\tilde{P}}^G(Q(\Delta_{1})\otimes\ldots\otimes Q(\Delta_s))$. Define $\tilde{M}_L = GL_{e_1}(L) \times \ldots \times GL_{e_s}(L)$ a Levi subgroup of a standard parabolic $\tilde{P}_L$ such that $B_L \subset\tilde{P}_L \subset G_L$.  In the same way as for diagram (\ref{D2}), we get a commutative diagram:
	$$\xymatrix{
		\mathcal{R}_{\lambda}(G)  \ar[r]^-{H}  &\mathcal{R}_{1}(G_{L})\\
		\mathcal{R}_{\lambda_{M}}(\tilde{M}) \ar[u]^{i_{\tilde{P}}^G} \ar[r]_-{H_{\tilde{M}}}  &\mathcal{R}_{1}(\tilde{M}_{L}),\ar[u]_-{i_{\tilde{P}_L}^{G_L}} } $$
	where the horizontal arrows are equivalences of categories, constructed in a similar fashion to the diagram (\ref{D2}) replacing $P$ by $\tilde{P}$, $B_L$ by $\tilde{P}_L$ and so on... Moreover by construction the functor $H_{\tilde{M}}=H_1 \times \ldots \times H_s$ is a product of functors $H_i$, where each individual functor $H_i : \mathcal{R}_{\lambda}(GL_{re_i}(F))  \longrightarrow \mathcal{R}_{1}(GL_{e_i}(L))$ is constructed in the same way as $H$, with $G$ replaced by $GL_{re_i}(F)$. Gathering all the results above we may write:
	$$H\circ i_{\tilde{P}}^G(Q(\Delta_{1})\otimes\ldots\otimes Q(\Delta_s))= i_{\tilde{P}_L}^{G_L}(H_{\tilde{M}}(Q(\Delta_{1})\otimes\ldots\otimes Q(\Delta_s)))$$
	$$=i_{\tilde{P}_L}^{G_L}(Q(\Delta'_{1})\otimes\ldots\otimes Q(\Delta'_s)),$$
	where $\Delta'_{i}=[1(\alpha_{i}),1(\alpha_{i}+e_{i}-1)]$ for all $i$. Hence we have an equality $H(\pi)=Q(\Delta'_{1},\ldots,\Delta'_{s})$, since both representations are the Langlands quotient of $H\circ i_{\tilde{P}}^G(Q(\Delta_{1})\otimes\ldots\otimes Q(\Delta_s))=i_{\tilde{P}_L}^{G_L}(Q(\Delta'_{1})\otimes\ldots\otimes Q(\Delta'_s))$.
\end{proof}

According to the description of Hecke algebras in section (5.6) of \cite{MR1204652} the isomorphism of Hecke algebras $\Phi:  \mathcal{H}(G, J, \lambda) \simeq \mathcal{H}(G_{L}, I_{L}, 1) $ is support preserving,  in the sense that $\mathrm{supp}(\Phi(f))=I_L.\mathrm{supp}(f).I_L$. We have also a natural isomorphism between $\mathcal{H}(K_{L}, I_{L}, 1)= \left\lbrace f \in \mathcal{H}(G_{L}, I_{L}, 1) \mid \mathrm{supp}(f) \subset K_{L}\right \rbrace$ and $\mathcal{H}(K, J, \lambda)= \left\lbrace f \in \mathcal{H}(G, J, \lambda) \mid \mathrm{supp}(f) \subset K\right \rbrace$. We have then the following commutative diagram:
$$\xymatrix{
	\mathcal{H}(G, J, \lambda)  \ar[r]^{\Phi} &\mathcal{H}(G_{L}, I_{L}, 1)\\
	\mathcal{H}(K, J, \lambda) \ar[u] \ar[r]^{\Phi} &\mathcal{H}(K_{L}, I_{L}, 1) \ar[u]}$$
As for diagram (\ref{D2}), the diagram above induces:
$$\xymatrix{
	\mathcal{R}_{\lambda}(G)  \ar[r]^-{M_{\lambda}} &\mathcal{H}(G, J, \lambda)\text{-Mod}  \ar[r]^{\Phi_{\ast}} &\mathcal{H}(G_{L}, I_{L}, 1)\text{-Mod}\ar[r]^-{T_{\lambda}} &\mathcal{R}_{1}(G_{L})\\
	\mathcal{R}_{\lambda}(K) \ar[u]^{\mathrm{c\text{--} Ind}_{K}^{G}} \ar[r]_-{M_{\lambda}} &\mathcal{H}(K, J, \lambda)\text{-Mod}  \ar[r]^{\Phi_{\ast}} &\mathcal{H}(K_{L}, I_{L}, 1)\text{-Mod} \ar[r]_-{T_{K_L}} &\mathcal{R}_{1}(K_{L})\ar[u]_-{\mathrm{c\text{--} Ind}_{K_{L}}^{G_{L}}}}  $$
\noindent where $T_{K_L}=\bullet \otimes_{\mathcal{H}(K_{L}, I_{L}, 1)}\mathrm{c\text{--} Ind}_{I_{L}}^{K_{L}} 1$.

If we denote the composition of all the top horizontal arrow by $H$ and the composition of all the bottom horizontal arrow by $H_{K}$, then $H(\mathrm{c\text{--} Ind}_{K}^{G} \sigma) = \mathrm{c\text{--} Ind}_{K_{L}}^{G_{L}} H_{K}(\sigma)$. 

The functor $\mathrm{Ind}_{J_{max}}^{K}(\kappa_{max} \otimes \cdot) : \mathscr{M}_{\nu}(\overline{G}) \to \mathcal{R}_{\lambda}(K) $ is an equivalence of categories according the discussion above Proposition 11 in Section 5 \cite{MR1728541}. We will denote this functor by $"\kappa_{max}"$. Let $\sigma_{\tilde{\mathcal{P}}}(\lambda)  = \mathrm{Ind}_{J_{max}}^{K}(\kappa_{max} \otimes \sigma(\tau,\tilde{\mathcal{P}}))$. 

\begin{lemma}\label{1.23}
	We have $H_{K}(\sigma_{min}(\lambda))=\mathrm{St}$, where $\mathrm{St}$ denotes the inflation of Steinberg representation of $GL_{n}$ over a finite field.
\end{lemma}

\begin{proof} By Lemma \ref{1.15} we have that $\overline{H}_e(st(1,e)) = \sigma(\tau,\tilde{\mathcal{P}}_{min})$. To conclude, use following commutative diagram:$$\xymatrix{
		\mathscr{M}_1(\overline{G}_e) \ar[d]^{"\kappa_{max}"} \ar[r]^{\overline{H}_e} &\mathscr{M}_{\nu}(\overline{G}) \ar[d]^{"\kappa_{max}"}\\
		\mathcal{R}_1(K_L)  \ar[r]^{H_K^{-1}} &\mathcal{R}_{\lambda}(K) \quad,}$$
	\noindent where every arrow is an equivalence of categories.
\end{proof}

\section{Generic representations}\label{M.11}

In this section we will use the results proven above to deduce our main theorem.


\begin{thm}\label{1.18}
	Let $\pi$ be an absolutely irreducible representation in the Berstein component $\Omega$, then $\mathrm{Hom}_{K}(\sigma_{min}(\lambda), \pi) \neq 0$ if and only if $\pi$ is generic.
\end{thm}

\begin{proof} Let us first deal with a particular case before the general case.
	
	\noindent \textbf{1. Simple type case.} Assume that $\pi$ contains a simple type $(J,\lambda)$. In this case $\Omega = [GL_{r}(F)^{e}, \omega \otimes \ldots \otimes \omega]_{G}$ where the tensor product $\rho : = \omega \otimes \ldots \otimes \omega$  is taken $e$ times and $\omega$ is a supercuspidal representation of $GL_{r}(F)$.

	Recall that representation $\pi=Q(\Delta_{1},\ldots,\Delta_{s})$  such that for $i < j$ the segment $\Delta_{i}$ does not precede $\Delta_{j}$.  If $s=1$ then $\pi$ is generic and contains $\sigma_{min}(\lambda)$. Assume that $s > 1$.

	The functor $H:\mathcal{R}_{\lambda}(G) \longrightarrow \mathcal{R}_{1}(G_{L})$ from above is an equivalence of categories. To avoid notational overload let $\sigma : = \sigma_{min}(\lambda)$, then 
	$$\mathrm{Hom}_{G}(\mathrm{c\text{--} Ind}_{K}^{G} \sigma,\pi)= \mathrm{Hom}_{G_{L}}(H(\mathrm{c\text{--} Ind}_{K}^{G} \sigma),H(\pi))$$

	Recall the following commutative diagram:
	$$\xymatrix{
		\mathcal{R}_{\lambda}(G)  \ar[r]^-{M_{\lambda}} &\mathcal{H}(G, J, \lambda)\text{-Mod}  \ar[r]^{\Phi_{\ast}} &\mathcal{H}(G_{L}, I_{L}, 1)\text{-Mod}\ar[r]^-{T_{\lambda}} &\mathcal{R}_{1}(G_{L})\\
		\mathcal{R}_{\lambda}(K) \ar[u]^{\mathrm{c\text{--} Ind}_{K}^{G}} \ar[r]_-{M_{\lambda}} &\mathcal{H}(K, J, \lambda)\text{-Mod}  \ar[r]^{\Phi_{\ast}} &\mathcal{H}(K_{L}, I_{L}, 1)\text{-Mod} \ar[r]_-{T_{K_L}} &\mathcal{R}_{1}(K_{L})\ar[u]_-{\mathrm{c\text{--} Ind}_{K_{L}}^{G_{L}}}}  $$
	\noindent where $T_{K_L}=\bullet \otimes_{\mathcal{H}(K_{L}, I_{L}, 1)}\mathrm{c\text{--} Ind}_{I_{L}}^{K_{L}} 1$.
	
	Recall that $H$ is the composition of all the top horizontal arrows and $H_K$ is the composition of all the bottom horizontal arrows, then $H(\mathrm{c\text{--} Ind}_{K}^{G} \sigma) = \mathrm{c\text{--} Ind}_{K_{L}}^{G_{L}} H_{K}(\sigma)$. By Lemma \ref{1.23} we have $H_{K}(\sigma)=\mathrm{St}$, where $\mathrm{St}$ denotes the inflation of Steinberg representation of $GL_{n}$ over a finite field. Moreover by Lemma \ref{1.22}, we have $H(\pi)=Q(\Delta'_{1},\ldots,\Delta'_{s})$. Therefore: 
	$$\mathrm{Hom}_{K}(\sigma,\pi|K)=\mathrm{Hom}_{G}(\mathrm{c\text{--} Ind}_{K}^{G} \sigma,\pi)= \mathrm{Hom}_{G_{L}}(H(\mathrm{c\text{--} Ind}_{K}^{G} \sigma),H(\pi))$$
	$$=\mathrm{Hom}_{G_{L}}(\mathrm{c\text{--} Ind}_{K_{L}}^{G_{L}} H_{K}(\sigma), Q(\Delta'_{1},\ldots,\Delta'_{s}))$$ $$=\mathrm{Hom}_{K_{L}}( \mathrm{St},Q(\Delta'_{1},\ldots,\Delta'_{s})|K_{L}).$$
	According to \cite[Theorem 9.7]{MR584084} $\pi = Q(\Delta_{1},\ldots,\Delta_{s})$ is generic if and only if no two segments $\Delta_i$ are linked. By construction the relative positions of the segments $\Delta_i$ are the same as of the segments $\Delta'_i$. Therefore no two segments $\Delta_i$ are linked if and only if no two segments $\Delta'_i$ are linked. It follows that $Q(\Delta'_{1},\ldots,\Delta'_{s})$ is generic if and only if $\pi$ is generic and $\mathrm{Hom}_{K_{L}}(\mathrm{St},Q(\Delta'_{1},\ldots,\Delta'_{s})|K_{L}) \neq 0$ if and only if $\mathrm{Hom}_{K}(\sigma,\pi|K) \neq 0$ by the equality above. So we are reduced to consider the case when $(J,\lambda)=(I,1)$. However this was proven in  \cite[section 7.2]{MR1915088}.
	
	\noindent \textbf{2. Semi-simple type case (general case).} Let now $\lambda$ be some general semi-simple type. The second part of the Main Theorem of section 8 in \cite{MR1643417} gives a support preserving Hecke algebra isomorphism $j :\mathcal{H}(\overline{M}, \lambda_{M}) \rightarrow \mathcal{H}(G, \lambda)$ (here $\overline{M}$ is the unique Levi subgroup of $G$ which contains the $N_{G}(M)$-stabilizer of the inertia class $D=[M,\rho]_M$ and is minimal for this property), and section 1.5 of \textit{op. cit.} gives a tensor product decomposition $\mathcal{H}(\overline{M}, \lambda_{M}) = \mathcal{H}_{1} \otimes_{C}\ldots\otimes_{C}\mathcal{H}_{s}$, where $\mathcal{H}_{i} = \mathcal{H}(G_{i}, J_{i}, \lambda_{i})$ is an affine Hecke algebras of type A and $(J_{i}, \lambda_{i})$ is some simple type with $G_{i}$ some general linear group over a $p$-adic field. 
	
	Let $M=\prod_{i=1}^{s} GL_{n_i}(F)$ be a standard block-diagonal Levi subgroup of a standard parabolic $P=MN$, such that $K \cap M = \prod_{i=1}^{s} K_{i}$, where $K_{i}$ is a maximal compact subgroup of $GL_{n_i}(F)$. By definition, see the end of section 6 in \cite{MR1728541}, the restriction of the $K$-representation $\sigma:=\sigma_{min}(\lambda)$ to $K \cap N$ is trivial, and $\sigma|K\cap M \simeq \sigma_{1}\otimes\ldots\otimes \sigma_{s}$ where $\sigma_{i} := \sigma_{\mathcal{P}_{i}^{min}}(\lambda_{i})$ with obvious notations.
	
	According to  \cite[Theorem (8.5.1)]{MR1204652} the irreducible representation $\pi$ is of the form
	$$\pi \simeq i_P^G(\pi_{1} \otimes \ldots \otimes \pi_{s}),$$
	\noindent such that $\pi_{i}$ is an irreducible representation of $G_{i}$ and contains the simple type $(J_{i}, \lambda_{i})$. Moreover the supercuspidal support of $\pi_{i}$ is disjoint from the supercuspidal support of $\pi_{j}$ for $i \neq j$. Then
	$$\mathrm{Hom}_{K}(\sigma,\pi) = \mathrm{Hom}_{K}(\sigma,\mathrm{Ind}_{K\cap P}^K(\pi_{1}|K_{1}\otimes\ldots\otimes \pi_{s}|K_{s}))$$ $$=\mathrm{Hom}_{K \cap P}(\sigma| K\cap P,\pi_{1}|K_{1}\otimes\ldots\otimes \pi_{s}|K_{s})$$
	$$ = \mathrm{Hom}_{K \cap M}(\sigma_{1}\otimes\ldots\otimes \sigma_{s},\pi_{1}|K_{1}\otimes\ldots\otimes \pi_{s}|K_{s}),$$
	where the first equality is obtained from the Mackey formula and Iwasawa decomposition, the second equality follows from Frobenius reciprocity, where $\pi_{1}|K_{1}\otimes\ldots\otimes \pi_{s}|K_{s}$ denotes the inflation of the representation of $K\cap M$ to $K\cap P$, and the last  equality is obtained by taking the coinvariants of $\sigma|K\cap P$ with respect to $K \cap N$. Hence $\mathrm{Hom}_{K}(\sigma,\pi)$ is non zero if and only $\mathrm{Hom}_{K_{i}}(\sigma_{i},\pi_{i}|K_{i})$ are non zero for all $i$. However, $\mathrm{Hom}_{K_{i}}(\sigma_{i},\pi_{i}|K_{i})$ are non zero for all $i$  if and only if $\pi_i$ are generic for all $i$ (by the simple type case for each $i$). Finally $\pi_i$ are generic for all $i$ if and only if $\pi$ is generic, because  the supercuspidal supports of $\pi_i$ are pairwise disjoint and all segment are pairwise disjoint. This finishes the proof.
\end{proof}
We may now deduce the multiplicity one statement:
\begin{lemma}\label{4.34}
	We have $\dim \mathrm{Hom}_K(\sigma_{min}(\lambda),\pi)=1$, for $\pi$ an irreducible generic representation of $G$ in $\Omega$.
\end{lemma}
\begin{proof} Let $x:=\mathfrak{m}_x \in \Spm \mathfrak{Z}_{\Omega}$ the maximal ideal defined by $\pi$ and $\kappa(x):= \mathfrak{Z}_{\Omega}/\mathfrak{m}_x$. Since $\pi$ is generic we have that $\mathrm{Hom}_K(\sigma_{min}(\lambda),\pi)\neq 0$ by Theorem \ref{1.18}. It follows that we have  $\cI_K^{G} \sigma_{min}(\lambda) \otimes_{\mathfrak{Z}_{\Omega}}\kappa(x) \twoheadrightarrow \pi$. Since the functor $\mathrm{Hom}_K(\sigma_{min}(\lambda),.)$ is exact, we have $\mathrm{Hom}_K(\sigma_{min}(\lambda),\cI_K^{G} \sigma_{min}(\lambda) \otimes_{\mathfrak{Z}_{\Omega}}\kappa(x)) \twoheadrightarrow \mathrm{Hom}_K(\sigma_{min}(\lambda),\pi)$. Moreover by Frobenius reciprocity we have that 
	$$\mathrm{Hom}_K(\sigma_{min}(\lambda),\cI_K^{G} \sigma_{min}(\lambda) \otimes_{\mathfrak{Z}_{\Omega}}\kappa(x))$$
	$$=\mathrm{Hom}_G(\cI_K^{G}\sigma_{min}(\lambda),\cI_K^{G} \sigma_{min}(\lambda) \otimes_{\mathfrak{Z}_{\Omega}}\kappa(x))$$ 
	\noindent and  by  \cite[Lemma 5.2]{Pyv1}:
	\[\mathrm{Hom}_G(\cI_K^{G} \sigma_{min}(\lambda),\cI_K^{G} \sigma_{min}(\lambda) \otimes_{\mathfrak{Z}_{\Omega}}\kappa(x)) \]
	\[\simeq \mathrm{Hom}_G(\cI_K^{G} \sigma_{min}(\lambda),\cI_K^{G} \sigma_{min}(\lambda)) \otimes_{\mathfrak{Z}_{\Omega}}\kappa(x)\]
	Since $\sigma_{min}(\lambda)$ occurs with multiplicity one in $\mathrm{Ind}_J^K \lambda$, then  by  \cite[Corollary 7.2]{Pyv1}, we have \[ \mathfrak{Z}_{\Omega} \simeq \mathrm{Hom}_G(\cI_K^{G} \sigma_{min}(\lambda),\cI_K^{G} \sigma_{min}(\lambda)).\] It follows that \[ \mathrm{Hom}_K(\sigma_{min}(\lambda),\cI_K^{G} \sigma_{min}(\lambda) \otimes_{\mathfrak{Z}_{\Omega}}\kappa(x)) \simeq \kappa(x).\]Hence we have a surjective map of $\kappa(x)$-vector spaces:	\[\kappa(x)\twoheadrightarrow \mathrm{Hom}_K(\sigma_{min}(\lambda),\pi)\]
	Then $1\geq \dim \mathrm{Hom}_K(\sigma_{min}(\lambda),\pi)$ and this space is non-zero, hence it must be one-dimensional.
\end{proof}


\subsection*{Acknowledgments}  The results of this paper are a part of the author's PhD thesis. The author is tremendously grateful to his advisor Vytautas Pa\v{s}k\={u}nas for sharing his ideas with the author and for many helpful discussions. We would like to thank Peter Schneider for pointing out the reference \cite{MR1915088}, which allowed to simplify some of our arguments. The author would also like to thank the referee for useful comments and corrections, which improved considerably the exposition of this paper. This work was supported by SFB/TR 45 of the DFG.

\bibliographystyle{alpha}
\addcontentsline{toc}{section}{References}
\bibliography{Monodromy}
\nocite{*}

\noindent Morningside Center of Mathematics, No.55 Zhongguancun Donglu, Academy of Mathematics and Systems Science, Beijing , Haidian District, 100190 China
\\
\textit{E-mail address}: pyvovarov@amss.ac.cn

\end{document}